\documentclass[11pt]{amsart} 

\usepackage{amsmath, amsthm, amssymb, fullpage}

\usepackage{enumerate, paralist}

\usepackage{hyperref}
\newcommand{\bC}{\mathbb{C}}
\newcommand{\bN}{\mathbb{N}}
\newcommand{\bR}{\mathbb{R}}
\newcommand{\rd}{\mathrm{d}}
\newcommand{\bP}{\mathbb{P}}
\newcommand{\FS}{\operatorname{FS}}
\newcommand{\Id}{\mathrm{Id}}
\DeclareMathOperator{\ord}{ord}
\DeclareMathOperator{\supp}{supp}
\DeclareMathOperator{\PGL}{PGL}
\DeclareMathOperator{\diam}{diam}
\DeclareMathOperator{\SAT}{SAT}
\DeclareMathOperator{\Int}{int}

\theoremstyle{plain}
\newtheorem{theorem}{Theorem}[section]

\newtheorem{mainth}{Theorem}
\newtheorem{maincoro}{Corollary}

\theoremstyle{definition}
\newtheorem{definition}[theorem]{Definition}

\newtheorem*{notation}{Notation}
\newtheorem*{acknowledgement}{Acknowledgement}
\newtheorem{remark}[theorem]{Remark}
\theoremstyle{remark}

\numberwithin{equation}{section}

\makeatletter
\@namedef{subjclassname@2020}{\textup{2020} Mathematics Subject Classification}
\makeatother

\begin{document}

\title[Selberg and Brolin]{Selberg and Brolin on value distribution of complex dynamics}

\author{Y\^usuke Okuyama}
\address{Division of Mathematics, Kyoto Institute of Technology, Sakyo-ku, Kyoto 606-8585 JAPAN}
\email{okuyama@kit.ac.jp}
 



\date{\today}

\begin{abstract}
The Brolin-Lyubich-Freire--Lopes--Ma\~n\'e equidistribution theorem for iterated preimages of a given non-exceptional value and Lyubich's periodic point version of it are foundational in the study of dynamics of rational functions of degree more than one on the complex projective line, and Drasin and the author studied a quantification of the former in a formalism of Nevanlinna theory, or more specifically, with the aid of Selberg's theorem from covering theory. In this paper, we point out that the argument in that previous study has already yielded a better quantification of the Brolin-Lyubich-Freire--Lopes--Ma\~n\'e equidistribution theorem, and also point out that a similar argument also yields a quantification of Lyubich's theorem under an exponentwise version of the so called hypothesis H,
which has recently been settled by Przytycki for unicritical polynomials. 
As an application, 
we sharpen the quantitative results on the value distribution of the (first) derivatives of iterated polynomials.
\end{abstract}

\maketitle

\section{Introduction}\label{sec:intro}

Let us recall the following classical theorem from the Nevanlinna theory or more specifically the 
Teichm\"uller-Selberg-Collingwood-Tsuji covering theory
(see e.g.\ \cite[Chapter V, Section 2]{Tsuji59}).

\begin{theorem}[Selberg {\cite[p.\ 311]{Selberg44}},
for a modern formulation, see {\cite{Weitsman72}}]\label{th:Selberg} Let $V$ be a bounded and at most finitely connected domain in $\bC$ whose boundary components are piecewise real analytic Jordan closed curves in $\bC$, 
so that for every $y\in V$, the (positive real valued) Green function $G_V(\cdot,y)$ on $V$ with pole $y$ exists and extends continuously to $\bC$ by setting $\equiv 0$ on $\bC\setminus V$. If $V\Subset\bC\setminus\{0\}$, 
then for every $y\in V$ and every $r>0$, setting 
the angular measure $\theta_V(r):=\int_{\{\theta\in[0,2\pi]:\,re^{i\theta}\in V\}}\rd\theta\in[0,2\pi]$ of $V$
with respect to the circle $\{|z|=r\}$ in $\bC$, we have
\begin{gather}
 \int_0^{2\pi}G_V(re^{i\theta},y)\frac{\rd\theta}{2\pi}
 \le\min\biggl\{\frac{\pi}{2}\tan\frac{\theta_V(r)}{4},
 \log\max\Bigl\{1,\frac{r}{\inf_{z\in V}|z|}\Bigr\}\biggr\}.\label{eq:Selberg}
\end{gather} 
\end{theorem}

In the applications below to preimages of a chordal open disk under a rational function, by slightly decreasing the radius of the disk, we may assume without loss of generality that its boundary circle contains no critical values of the rational function.

In \cite{DOproximity}, Drasin and the author applied Theorem \ref{th:Selberg} to quantify the Brolin-Lyubich-Freire--Lopes--Ma\~n\'e equidistribution theorem on the iterated preimages of complex dynamics. Quite recently, this kind of quantitative study is getting more and more active. 

In this paper, we first point out that our argument in \cite{DOproximity} has already yielded a better quantification of the Brolin-Lyubich-Freire--Lopes--Ma\~n\'e equidistribution theorem than that stated in \cite{DOproximity}. Next, we also point out that a similar argument based on Selberg's theorem (Theorem \ref{th:Selberg}) also yields a quantification of Lyubich's theorem on distribution of periodic points of complex dynamics.

Let $f$ be a rational function on $\bP^1=\bP^1(\bC)$ of degree $d>1$.

\begin{notation}
 Let $[z,w]_{\bP^1}$ denote the chordal metric on $\bP^1$, and $\omega_{\FS}$ the Fubini-Study volume form on $\bP^1$, which are normalized so that the diameter of $(\bP^1,[z,w]_{\bP^1})$ equals $[0,\infty]_{\bP^1}=1$ (regarding $\bP^1$ as the Riemann sphere $\bC\sqcup\{\infty\}$) and $\omega_{\FS}(\bP^1)=1$, so that, concretely,
\begin{gather*}
 [z,w]_{\bP^1}=\frac{|z-w|}{\sqrt{1+|z|^2}\sqrt{1+|w|^2}}=|z-w|[z,\infty]_{\bP^1}[w,\infty]_{\bP^1}\quad\text{and}\quad\omega_{\FS}=\frac{r\rd r\rd\theta}{\pi(1+r^2)^2}\,\,(z=re^{i\theta}),
\end{gather*}
and $\diam$ denotes the chordal diameter function on $2^{\bP^1}$; we note that 
\begin{gather*}
 \Bigl|\int_{\bP^1}\log[\cdot,z]_{\bP^1}\omega_{\FS}(z)\Bigr|\equiv :C_\omega>0\quad\text{on }\bP^1,
\end{gather*}
and also set $C_f:=\sup_{\bP^1}((f^*\omega_{\FS})/\omega_{\FS})>0$. 
The canonical equilibrium measure of $f$ is the weak limit
\begin{gather*}
 \mu_f=\lim_{n\to\infty}\frac{(f^n)^*\omega_{\FS}}{d^n}\quad\text{on }\bP^1,
\end{gather*}
which is in fact the unique maximal entropy measure of $f$ and has no atoms. The Julia set of $f$, which coincides with $\supp\mu_f$,
is denoted by $J(f)$, and the complement of $J(f)$, the Fatou set of $f$, by $F(f)$.

For every $a\in\bP^1$ and every $s\in(0,1]$, set $B[a,s]:=\{z\in\bP^1:[z,a]_{\bP^1}<s\}$, as in \cite[p.~603]{DOproximity}.
Similarly, for every $a\in\bC$ and every $s>0$,
set $B(a,s):=\{z\in\bC:|z-a|<s\}$. 
\end{notation}

\subsection{Quantitative equidistribution of iterated preimages}
Our first principal result is the following quantitative mean proximity estimate of the iteration sequence $(f^n)$ to constant targets $a\in\bP^1$.

\begin{mainth}\label{th:proximity}
 Let $f\in\bC(z)$ be a rational function on $\bP^1=\bP^1(\bC)$ of degree $d>1$. Then for every $a\in\bP^1$ but superattracting periodic points of $f$, we have
\begin{gather}
 m(f^n,a):=\int_{\bP^1}\log[f^n(z),a]_{\bP^1}\omega_{\FS}(z)=O(n)
\quad\text{as }n\to\infty,\label{eq:proximity}
\end{gather}
where the implicit constant in $O$ in \eqref{eq:proximity} is locally uniform on any $a\in\bP^1$ but
superattracting periodic points of $f$. 
\end{mainth}
In the setting in Theorem \ref{th:proximity},
for any $C^2$-test function $\phi$ on $\bP^1$, 
by integration by parts, the proximity estimate \eqref{eq:proximity} also yields the order estimate of the value distribution
\begin{gather}
\biggl|\int_{\bP^1}\phi\Bigl(\frac{(f^n)^*\delta_a}{d^n}-\mu_f\Bigr)\biggr|\le\Bigl(\sup_{\bP^1}\Bigl|\frac{\rd\rd^c\phi}{\omega_{\FS}}\Bigr|\Bigr)\cdot\frac{|m(f^n,a)|+C_fC_\omega}{d^n}
=O(nd^{-n})\quad\text{as }n\to\infty\tag{\ref{eq:proximity}'}\label{eq:order}
\end{gather}
of the iteration sequence $(f^n)$ of $f$ for the constant target $a$. 
This order estimate \eqref{eq:order} is a quantification of the above mentioned equidistribution theorem towards $\mu_f$ of iterated preimages of points, 
and can be stated in a more general form as 
the $O((nd^{-n})^{\alpha/2})$ order estimate
for $C^\alpha$-test functions $\phi$ on $\bP^1$ where
$\alpha\in(0,2]$ with the aid of the Banach space interpolation theory. 

The weaker $O((\eta/d)^n)$ order estimate of the left hand side of \eqref{eq:order} (in the setting of Theorem \ref{th:proximity}) was established in \cite[as a consequence of Theorem 2]{DOproximity} where $\eta>1$, and soon after \cite{DOproximity} generalized to holomorphic endomorphisms of $\bP^N$ by Dinh--Sibony \cite{DS10} (see also Taflin \cite{Taflin11adv} and Ahn \cite{Ajn16} for further developments for positive closed currents, and also \cite{Sodin92, RS95, RS97} for (non-autonomous) sequence of rational maps between projective spaces). It might be of great interest whether the $O(nd^{-n})$ order estimate \eqref{eq:order} could also be generalized to holomorphic endomorphisms of $\bP^N$.

\begin{remark}
For any superattracting periodic point $a\in\bP^1$ of $f$, in \cite{DOproximity} we already gave the optimal order estimate of the left hand side of \eqref{eq:proximity} (and that of \eqref{eq:order}), so we do not touch this case in this paper. 
\end{remark}

\subsection{Quantitative equidistribution of periodic points}
Let us focus on the periodic points in $\bP^1$ under $f$.
Our next principal result is the following quantitative mean proximity estimate of the iteration sequence $(f^n)$ to the (moving) target $\Id_{\bP^1}$. 

\begin{mainth}\label{th:periodic}
 Let $f\in\bC(z)$ be a rational function on $\bP^1=\bP^1(\bC)$ of degree $d>1$ and pick any $\eta>1$, and 
suppose that $f$ satisfies the exponentwise version of the hypothesis H for some $\eta'\in(1,\eta^{1/3})$
stated below. Then 
\begin{gather}
 m(f^n,\Id_{\bP^1}):=\int_{\bP^1}\log[f^n,\Id_{\bP^1}]_{\bP^1}\omega_{\FS}
=O(\eta^n)
\quad\text{as }n\to\infty.\label{eq:proximperi}
\end{gather}
\end{mainth}
We note that we can write $O(\eta^n)$ as $o(\eta^n)$ in the statement of Theorem \ref{th:periodic}.

\begin{notation}
 Let us denote by $[f^n=\Id_{\bP^1}]$ the effective divisor on $\bP^1$ defined by the equation $f^n=\Id_{\bP^1}$, which is regarded as the sum of the Dirac measures 
 at all the fixed points $w\in\bP^1$ of $f^n$ taking into account the order $\ord_w[f^n=\Id_{\bP^1}]\in\bN$
 of the divisor $[f^n=\Id_{\bP^1}]$ at each $w$,
 so that the total mass of $[f^n=\Id_{\bP^1}]$ equals $d^n+1$. 
\end{notation}

The exponentwise version of the hypothesis H for each exponent $\eta>1$ (see Przytycki--Rivera-Letelier--Smirnov \cite{PRS04} for the original), which gets stronger and stronger as the exponent $\eta>1$ is closer and closer to $1$, prospects that 
{\em if $n\gg 1$, then for every subset $E$
in $\supp[f^n=\Id_{\bP^1}]$ such that
\begin{gather*}
 \max_{(z,w)\in E^2}d_{\infty,n-1}\bigl((f^j(z))_{j=0}^{n-1},(f^j(w))_{j=0}^{n-1}\bigr)<\eta^{-n}, 
\end{gather*}
we have $\#E\le\eta^n$,}
where $d_{\infty,n-1}$ is the supremum metric on $(\bP^1)^n$ induced by the chordal metric on $\bP^1$
(see Remark \ref{th:hypH} for some discussion).

Again, in the setting in Theorem \ref{th:periodic},
for any $C^2$ test function $\phi$ on $\bP^1$, 
by integration by parts, the proximity estimate
\eqref{eq:proximperi} (under the exponentwise version of the hypothesis H for 
some $\eta'\in(1,\eta^{1/3})$)
yields the following order estimate of the equidistribution 
\begin{gather}
\biggl|\int_{\bP^1}\phi\Bigl(\frac{[f^n=\Id_{\bP^1}]}{d^n+1}-\mu_f\Bigr)\biggr|
\le\Bigl(\sup_{\bP^1}\Bigl|\frac{\rd\rd^c\phi}{\omega_{\FS}}\Bigr|\Bigr)\cdot\frac{|m(f^n,\Id_{\bP^1})|+C_fC_\omega}{d^n+1}
=O\Bigl(\Bigl(\frac{\eta}{d}\Bigr)^n\Bigr)\quad\text{as }n\to\infty\tag{\ref{eq:proximperi}'}\label{eq:equiperi}
\end{gather}
of periodic points under $f$ towards $\mu_f$ (see also Naeger \cite{Naeger20} in the case that $f$ is a quadratic polynomial having a 
(super) attracting fixed point in $\bC$).
This order estimate \eqref{eq:equiperi} is a quantification of Lyubich's periodic point version (\cite{Lyubich83}, see also Tortrat \cite{Tortrat87} when $f$ is a polynomial) of Brolin's theorem, and can be stated in a more general form as the $O((\eta/d)^{n\alpha/2})$ order estimate for $C^\alpha$-test functions $\phi$ on $\bP^1$ where $\alpha\in(0,2]$ with the aid of the Banach space interpolation theory.

For now it has been known (in Przytycki \cite[Theorem 5.1 and p.~5]{Przytycki25}) that
when $P$ is a unicritical polynomial on $\bC$,
i.e., when $P$ is $\PGL(2,\bC)$-conjugate to 
$z^d+\lambda$ for some $\lambda\in\bC$,
$P$ satisfies the exponentwise version of the hypothesis H for any $\eta>1$.
Hence the following holds unconditionally.
\begin{maincoro}
For every unicritical polynomial $P(z)$ on $\bC$
of degree $d>1$,
every $\eta>1$, and every $C^2$ test function $\phi$ on $\bP^1$, we have
 \begin{gather}
 \int_{\bP^1}\phi\Bigl(\frac{[P^n=\Id_{\bP^1}]}{d^n+1}-\mu_P\Bigr)=O\Bigl(\Bigl(\frac{\eta}{d}\Bigr)^n\Bigr)\quad\text{as }n\to\infty.
\end{gather}
\end{maincoro}
For a similar order estimate of the left hand side in \eqref{eq:equiperi} for $C^1$ test functions on $\bP^1$ based on a different idea (e.g., a use of Cauchy-Schwarz inequality), we refer to recent works of Dinh--Kaufmann and de Th\'elin--Dinh--Kaufmann \cite{DDK25} and of Gauthier--Vigny \cite{GV25}.

In Section \ref{sec:proof}, we give a proof outline of Theorem \ref{th:proximity} by an argument similar to that in \cite{DOproximity}, which contains a key covering theory computation based on Theorem \ref{th:Selberg} (Selberg's theorem). In Section \ref{sec:periodic}, after preparing some conclusion under the exponentwise version of the hypothesis H in \cite{PRS04},
we give a proof of Theorem \ref{th:periodic}. In Section \ref{sec:derivatives}, we sharpen the quantitative results on the value distribution of the (first) derivatives of iterated polynomials 
from \cite{OkuyamaDerivatives} since Theorem \ref{th:proximity} is applied efficiently there and this kind of research is also getting active in the context of orthogonal and extremal polynomials (see e.g., \cite{Totik19, HPU24, HPU24centering, HPU26}).

\begin{acknowledgement}
This research was partially supported by JSPS Grant-in-Aid for Scientific Research (C), 23K03129.
\end{acknowledgement}

\section{Proof of Theorem \ref{th:proximity}}
\label{sec:proof}

Let $f\in\bC(z)$ be a rational function on $\bP^1=\bP^1(\bC)$ of degree $d>1$. 
The critical set of $f$ is
\begin{gather*}
 C(f):=\{c\in\bP^1:\deg_c f>1\}, 
\end{gather*}
which consists of at most $2d-2$ points
since $\sum_{z\in\bP^1}(\deg_z f-1)=2d-2$ by Riemann-Hurwitz. 

For every $a\in\bP^1$, every $s\in(0,1]$, and every $n\in\bN\cup\{0\}$,
$V^{-n}_{a,s}$ denotes any component of 
$f^{-n}(B[a,s])$.

\begin{remark}\label{th:Julia}
Since the case that $a\in\bP^1$ is in the Fatou set $F(f)$ of $f$ is already done in \cite[\S 3]{DOproximity}, 
we will focus on the case that
 $a$ is in the Julia set $J(f)$ of $f$. Recall that $\supp\mu_f=J(f)$.
\end{remark}

By Denker--Przytycki--Urbanski \cite[Lemma 3.4]{DPU96}
(based on \cite[Lemma 2.3 (named as Rule II there)]{DPU96}), 
there are $L_1\ge 1$ and $\rho\in(0,1)$ such that, if $0<s\ll 1$, then for every $n\in\bN\cup\{0\}$,
\begin{gather}
 \sup_{a\in J(f)}\Bigl(\max_{V_{a,s}^{-n}}
\diam(V_{a,s}^{-n})\Bigr)\le L_1^ns^\rho.
\end{gather}
We note that
\begin{gather*}
 \inf_{c,c'\in C(f):\,c\ne c'}[c,c']_{\bP^1}>0\quad\text{and}\quad
\inf_{c\in C(f)\cap F(f)\text{ and }z\in J(f)}[c,z]_{\bP^1}>0. 
\end{gather*}
Moreover, by Przytycki \cite[\S 2]{Przytycki93} (see also \cite[Lemma 2.2 (named as Rule I there, which is a weaker version of Rule II)]{DPU96}), there is $L_2\in(0,1]$ such that for every $n\in\bN$ and every $c\in C(f)\cap J(f)$,
\begin{gather*}
 [c,f^n(c)]_{\bP^1}\ge L_2^n.
\end{gather*}
Pick any sequence $(s_n)_{n=0}^\infty$ 
in $\bR_{>0}$ satisfying the (upper and lower) decay condition
\begin{gather}
-\infty<\liminf_{n\to\infty}\frac{\rho\log s_n}{n-1}\le
 \limsup_{n\to\infty}\frac{\rho\log s_n}{n-1}
<-\log\frac{L_1}{L_2}(\le 0),\label{eq:decay} 
\end{gather}
so that, in particular, by the upper one, if $n\gg 1$, then for every $a\in J(f)$ and every $j\in\{1,\ldots,n\}$,
\begin{gather}
\#\bigl(V_{a,s_n}^{-j}\cap C(f)\bigr)\le 1.\label{eq:critical}
\end{gather}
We claim that if $n\gg 1$, then for every $a\in J(f)$ and every $c\in C(f)$, we also have
\begin{gather}
\#\bigl\{k\in\{0,\ldots,n-1\}:c\in f^k(V_{a,s_n}^{-n})\bigr\}\le 1;\label{eq:criticalorbit}
\end{gather}
indeed, if $n\gg 1$, then for any $a\in J(f)$ and 
any $k,\ell\in\{1,\ldots,n\}$, $k>\ell$, 
by the upper decay condition on $s_n$ in \eqref{eq:decay}, we have
\begin{gather*}
\diam(V_{a,s_n}^{-\ell})\le L_1^\ell s_n^\rho
\le L_1^{n-1} s_n^\rho< L_2^{n-1}\le L_2^{k-\ell}
\quad\text{for any }V_{a,s_n}^{-\ell}, 
\end{gather*}
which yields that for any $c\in C(f)\cap J(f)$,
$\{c,f^{k-\ell}(c)\}\not\subset V_{a,s_n}^{-\ell}$, 
so $c\not\in V_{a,s_n}^{-k}\cap V_{a,s_n}^{-\ell}$ 
if $f^{k-\ell}(V_{a,s_n}^{-k})=V_{a,s_n}^{-\ell}$. 
Hence the claim holds.

\begin{remark}
In \cite{DOproximity}, instead of the counting \eqref{eq:criticalorbit}, we were based on a weaker 
\begin{gather*}
\limsup_{n\to\infty}\frac{\sup_{a\in J(f)}\sup_{V_{a,s_n}^{-n}}\#\bigl\{j\in\{0,\ldots,n-1\}:f^j(V_{a,s_n}^{-n})\cap C(f)\neq\emptyset\bigr\}}{n}=0,
\end{gather*}
which was a consequence of \cite[Lemma 0]{Przytycki90}. 
\end{remark}
Once the sharper counting \eqref{eq:criticalorbit} is at our disposal, we are done by the same argument as that in \cite[Section 3, Case 2]{DOproximity} based on Theorem \ref{th:Selberg} (the Selberg theorem). For completeness, we include a proof assuming standard facts from complex dynamics (e.g., the density of repelling periodic points of $f$ in $J(f)$). Recall that $f$ is a Lipschitz endomorphism of $(\bP^1,[z,w]_{\bP^1})$ and that any projective transformation in $\PGL(2,\bC)$ is a biLipschitz automorphism of $(\bP^1,[z,w]_{\bP^1})$.

\begin{proof}[Proof of Theorem \ref{th:proximity}]
Pick $a\in J(f)$.
By $\#C(f)<\infty$ and the local expressions of $f$ along backward orbits of $a$, replacing $f$ with some $\PGL(2,\bC)$-conjugation of some iteration of $f$ if necessary, 
we assume that 
\begin{gather*}
 f(p)=p\quad\text{for each }p\in\{0,\infty\}
\end{gather*}
and 
that $a\in\bP^1\setminus\{0,\infty\}$, 
without loss of generality. Then for $0<R\ll 1$, 
there is $C\in(0,1]$ such that 
\begin{gather*}
 ([z,w]_{\bP^1}\le)|z-w|\le\frac{[z,w]_{\bP^1}}{C}
\quad\text{on }B[a,R]\times B[a,R],
\end{gather*}
and on the other hand,
if $n\gg 1$, then for every $a'\in B[a,R]$ and every $j\in\{1,\ldots,n\}$, 
we have not only
\begin{gather}
 V_{a',s_n}^{-j}\Subset\bP^1\setminus\{0,\infty\}\label{eq:cylinder} 
\end{gather}
but also $\#(V_{a',s_n}^{-j}\cap C(f))\le 1$, 
using the sharper counting \eqref{eq:critical} if $a'\in J(f)$;
the case of $a'\in F(f)$ has already been treated in \cite[Lemma 2.3]{DOproximity}.
Hence by the Riemann-Hurwitz formula, 
every $V_{a',s_n}^{-j}$ is a topological disk,
and in turn by \eqref{eq:criticalorbit}, 
 \begin{gather}
 \max_{V_{a',s_n}^{-n}}\deg\bigl(f^n:V_{a',s_n}^{-n}\to B[a',s_n]\bigr)\le d^{\#C(f)}<\infty.\label{eq:quasisemihyp}
 \end{gather}
By the Lipschitz continuity of $f$ and $f(0)=0$, 
there is also $C_1>0$ such that for every $n\in\bN$ and every
$z\in f^{-n}(B[a,R])$,
\begin{gather*}
 0<[a,0]_{\bP^1}-R \le[f^n(z),0]_{\bP^1}=[f^n(z),f^n(0)]_{\bP^1}\le e^{C_1n/2}[z,0]_{\bP^1}
 \le e^{C_1n/2}|z|,
\end{gather*}
so that for $n\gg 1$,
\begin{gather}
 \frac{1}{\min_{V_{a,R}^{-n}}
\inf_{z\in V_{a,R}^{-n}}|z|}
 \le\frac{e^{C_1n/2}}{[a,0]_{\bP^1}-R}
 \le e^{C_1n}.\label{eq:big}
\end{gather}

Now we follow a standard argument from the covering theory. For $n\gg 1$, every $a'\in B[a,R/2]$, 
and every component $V^{-n}$ of $f^{-n}(B(a',s_n))$,
 we have (cf. \cite{Myrberg33})
 \begin{gather}
 \log\frac{s_n}{|f^n(\cdot)-a'|}
 =\sum_{b\in V^{-n}\cap f^{-n}(a')}\deg_b(f^n)\cdot G_{V^{-n}}(\cdot,b)\quad\text{on }V^{-n}\label{eq:Myrberg}
 \end{gather}
 (and recall that $\deg(f^n:V^{-n}\to B(a',s_n))=\sum_{b\in V^{-n}\cap f^{-n}(a')}\deg_b(f^n)$),
 and moreover for every $r>0$, fixing $V^{-n}(r)$ among all $V^{-n}$ so that $\theta_{V^{-n}(r)}(r)=\max_{V^{-n}}\theta_{V^{-n}}(r)$, for every $V^{-n}$ but $V^{-n}(r)$, we have
 \begin{gather}
 \theta_{V^{-n}}(r)\le\pi\Bigl(=\frac{2\pi}{2}\Bigr). 
 \end{gather}
On one hand, for $n\gg 1$, by the former case in the estimate in Theorem \ref{th:Selberg} together with \eqref{eq:Myrberg}
 and \eqref{eq:quasisemihyp}, we have
 \begin{align*}
 & \sum_{V^{-n}:\,\theta_{V^{-n}}(r)\le\pi}\int_{V^{-n}\cap\{|z|=r\}}\log\frac{s_n}{|f^n(re^{i\theta})-a'|}\frac{\rd\theta}{2\pi}\\
 \le&\sum_{V^{-n}:\,\theta_{V^{-n}}(r)\le\pi}
 \deg(f^n:V^{-n}\to B(a',s_n))
 \cdot\frac{\pi}{2}\tan\Bigl(\frac{\theta_{V^{-n}}(r)}{4}\Bigr)\\ 
 \le& d^{\#C(f)}\cdot\frac{\pi}{2}\sum_{V^{-n}:\,\theta_{V^{-n}}(r)\le\pi}\frac{\theta_{V^{-n}}(r)}{\pi}
 \le d^{\#C(f)}\cdot\frac{\pi}{2}\cdot\frac{2\pi}{\pi}
=\pi\cdot d^{\#C(f)}\quad\text{for every }r>0.
 \end{align*}
 On the other hand, for $n\gg 1$, 
 by the latter case in the estimate in Theorem \ref{th:Selberg} together with \eqref{eq:big}, we have
 \begin{multline*}
 \int_{V^{-n}(r)\cap\{|z|=r\}}\log\frac{s_n}{|f^n(re^{i\theta})-a'|}\frac{\rd\theta}{2\pi}\\
 \le d^{\#C(f)}\log\max\Bigl\{1,\frac{r}{\inf_{z\in V^{-n}(r)}|z|}\Bigr\}
 \le d^{\#C(f)}\cdot\log\max\{1,e^{C_1n}r\}
\quad\text{for every }r>0. 
 \end{multline*}
Consequently, for $n\gg 1$, we have
 \begin{multline*}
 \int_{\bP^1}\log\frac{1}{[f^n(z),a']_{\bP^1}}\omega_{\FS}(z)+\log(C\cdot s_n)
\le\int_{f^{-n}(B(a',s_n))}\log\frac{s_n}{|f^n(z)-a'|}\omega_{\FS}(z)\\
 =\int_0^\infty\frac{2r\rd r}{(1+r^2)^2}
 \int_{f^{-n}(B(a',s_n))\cap\{|z|=r\}}\log\frac{s_n}{|f^n(z)-a'|}\frac{\rd\theta}{2\pi}\\
\le\int_0^\infty\frac{2r}{(1+r^2)^2}
\bigl(\pi\cdot d^{\#C(f)}+d^{\#C(f)}\cdot(C_1n+|\log r|)\bigr)\rd r 
=d^{\#C(f)}\cdot\biggl(\pi+C_1n+
\int_0^\infty\frac{2r|\log r|}{(1+r^2)^2}\rd r\biggr),
 \end{multline*}
which with the lower decay condition on $s_n$ 
in \eqref{eq:decay} concludes the proof.
\end{proof}

\section{Proof of Theorem \ref{th:periodic}}
\label{sec:periodic}
Let $f\in\bC(z)$ be a rational function on $\bP^1=\bP^1(\bC)$ of degree $d>1$.
Replacing $f$ with some $\PGL(2,\bC)$-conjugation of $f$ if necessary,
we assume that  
\begin{gather*}
\text{for each }p\in\{0,\infty\},\, 
p \text{ is strictly preperiodic, i.e., not periodic but preperiodic, under } f
\end{gather*}
without loss of generality.
We continue to use the notation $V^{-n}_{a,s}$ in Section \ref{sec:proof}, $a\in\bP^1$ and $s\in(0,1]$,
and for each $z\in f^{-n}(B[a,s])$, let $V^{-n}_{a,s}(z)$ be
the unique $V^{-n}_{a,s}$ containing $z$, so that 
\begin{gather}
 f^{n-j}\bigl(V^{-n}_{a,s}(z)\bigr)=V^{-j}_{a,s}\bigl(f^{n-j}(z)\bigr)\label{eq:proper}
\end{gather}
for every $j\in\{0,1,\ldots, n\}$. 
Let us denote by $L>1$ the Lipschitz constant of the endomorphism $f$ of $(\bP^1,[z,w]_{\bP^1})$ and recall that $f$ has at most finitely many non-repelling cycles.

\begin{definition}
 Following \cite[Definition 3.8]{PRS04}, for each $s\in(0,1]$ and each $n\in\bN$, we say a point $z\in J(f)\cap\supp[f^n=\Id]$ to be $s$-{\em regular} if the restriction of $(f^n)|(V_{f^j(z),s}^{-n}(f^j(z)))$ is univalent for some $j=j_z=j_{z,n}\in\{0,\ldots,n-1\}$. 
\end{definition}

For every $s\in(0,1]$, by so called Fatou's argument, if $n\gg 1$, then for any $s$-regular point $z\in J(f)\cap\supp[f^n=\Id]$, we have $V_{f^{j_z}(z),s/2}^{-n}(f^{j_z}(z))\subset B[f^{j_z}(z),s/4]$; then by Schwarz's lemma, we have
$V_{f^{j_z}(z),s/2}^{-n}(f^{j_z}(z))\cap(\supp[f^n=\Id])=\{f^{j_z}(z)\}$
and in turn have 
\begin{gather}
 V_{f^{j_z}(z),s/2}^{-n-j_z}(z)\cap\bigl(\supp[f^n=\Id]\bigr)=\{z\},\label{eq:simple} 
\end{gather}
and by \eqref{eq:proper} and the Lipschitz continuity of $f$, 
we also have
\begin{gather*}
[w,f^n(w)]_{\bP^1}\ge
\Bigl(\frac{s}{2}-\frac{s}{4}\Bigr)/L^{j_z}\ge\frac{s}{4L^n}=:s_n
\end{gather*}
for every $w\in\partial(V_{f^{j_z}(z),s/2}^{-n-j_z}(z))$. 
Hence denoting by $W_{s_n}^{(n)}(z)$ the component
of $\{w\in\bP^1:[f^n(w),w]_{\bP^1}<s_n\}$ containing this $z$, we have
\begin{gather}
 W_{s_n}^{(n)}(z)\subset V_{f^{j_z}(z),s/2}^{-n-j_z}(z).\label{eq:smaller}
\end{gather}
Pick $\eta>1$, and suppose that $f$ satisfies
the hypothesis H for some $\eta'\in(1,\eta^{1/3})$. Then for any $\eta''\in((\eta')^3,\eta)$, diminishing $s\in(0,1]$ if necessary, by \cite[Proof of (3.1)]{PRS04}, we have
\begin{gather}
 \sum_{z\in J(f)\cap(\supp[f^n=\Id_{\bP^1}]):\,\text{not }s\text{-regular}}\ord_z[f^n=\Id_{\bP^1}]=O(\eta''^n)\quad\text{as }n\to\infty.\label{eq:irregular}
\end{gather}
For each $p\in\{0,\infty\}$, every $n\in\bN$, 
and every $z\in\bC$, if $|f^n(z)-z|<s_n$, then
\begin{gather*}
 0<\inf_{m\in\bN}[f^m(p),p]_{\bP^1}
 \le[f^n(p),f^n(z)]_{\bP^1}+[f^n(z),z]_{\bP^1}+[z,p]_{\bP^1}
 <(L^n+1)[z,p]_{\bP^1}+s_n,
\end{gather*}
so that for every component $V^{(n)}$ of $\{|f^n-\Id_{\bP^1}|<s_n\}$, we have both
\begin{gather}
 V^{(n)}\Subset\bC\setminus\{0\}\quad\text{if }n\gg 1\tag{\ref{eq:cylinder}'}
\end{gather}
and
\begin{gather*}
 \log\frac{1}{\min_{V^{(n)}}\inf_{z\in V^{(n)}}|z|}=O(n)\quad\text{as }n\to\infty, 
\end{gather*}
and by \eqref{eq:simple}, \eqref{eq:smaller}, and \eqref{eq:irregular}, we also have
\begin{gather*}
 \max_{V^{(n)}}\sum_{z\in V^{(n)}\cap(\supp[f^n=\Id_{\bP^1}])}\ord_z[f^n=\Id_{\bP^1}]=O(\eta''^n)\quad\text{as }n\to\infty.
\end{gather*}

Now we follow a standard argument from the covering theory, similarly to Section \ref{sec:proof}, but the argument is simpler if we use Theorem \ref{th:proximity}. 
For $n\gg1$, by \eqref{eq:Myrberg} applied to $f^n-\Id_{\bP^1}$, we have
\begin{gather*}
 \log\frac{s_n}{|f^n-\Id_{\bP^1}|}
 =\sum_{z\in V^{(n)}\cap(\supp[f^n=\Id_{\bP^1}])}\ord_z[f^n=\Id_{\bP^1}]\cdot G_{V^{(n)}}(\cdot,z)
 \quad\text{on each }V^{(n)},
\end{gather*}
so that, under the strict preperiodicity assumption on $\infty$, we compute
\begin{multline*}
 \int_{\bP^1}\log\frac{1}{[f^n,\Id_{\bP^1}]_{\bP^1}}\omega_{\FS}+\log s_n
 \le\int_{\{|f^n-\Id_{\bP^1}|<s_n\}}\log\frac{s_n}{|f^n-\Id_{\bP^1}|}\omega_{\FS}
 -m(f^n,\infty)+C_\omega\\
 \le O(\eta''^n)\cdot\biggl(\frac{\pi}{2}\cdot\frac{2\pi}{\pi}+\Bigl(O(n)+
 \int_0^\infty\frac{2r|\log r|}{(1+r^2)^2}\rd r\Bigr)\biggr)-m(f^n,\infty)+C_\omega
 =O(n\eta''^n)+O(n)+C_\omega\quad\text{as }n\to\infty 
\end{multline*}
using the Selberg theorem (Theorem \ref{th:Selberg}) 
(see the final computation in the proof of Theorem \ref{th:proximity})
and Theorem \ref{th:proximity} applied to $\infty$.
Since $s_n=s/(4L^n)$ and $\eta''<\eta$, we have
$m(f^n,\Id_{\bP^1})=O(n\eta''^n)=o(\eta^n)$ as $n\to\infty$.
Now the proof of Theorem \ref{th:periodic} is complete. \qed

\begin{remark}\label{th:hypH}
The exponentwise version of the hypothesis H for any $\eta>1$ is free under some possibly non-uniform hyperbolicity conditions on $f$, e.g., the topological Collet-Eckmann condition on $J(f)$ or equivalently the uniform hyperbolicity condition on periodic points in $J(f)$ (see \cite{PRS03} for more details).
\end{remark}

\section{Quantitative value distribution of the derivatives of iterated polynomials}
\label{sec:derivatives}

In \cite{OkuyamaDerivatives}, we established a weaker $O(\eta^nd^{-n})$ version of the following quantitative value distribution of the sequence of (the first) derivatives of iterated polynomials by applying our former quantitative proximity estimate \cite[Theorem 2]{DOproximity}. 

\begin{theorem}\label{th:derivative}
Let $P\in\bC[z]$ be a polynomial of degree $d>1$.
Then there is a polar subset $E=E_P$ in $\bC$ such that
for every $a\in\bC\setminus E$
and every $C^2$-test function $\phi$ on $\bP^1=\bP^1(\bC)$,
\begin{gather*}
\int_{\bP^1}\phi\rd\biggl(\frac{((P^n)')^*\delta_a}{d^n-1}-\mu_P\biggr)=O(n^2d^{-n})\quad\text{as }n\to\infty.
\end{gather*}
\end{theorem}

Once Theorem \ref{th:proximity} is at our disposal, we have Theorem \ref{th:derivative}
using it in the place of \cite[Theorem 2]{DOproximity}. For completeness,
we herewith include a proof. Indeed, a polar exceptional set $E$ 
in Theorem \ref{th:derivative}
can be the set of poles in $\bC$ of the $(\sum_{n=1}^{\infty}1/n^2)\omega_{\FS}$-subharmonic function
\begin{gather*}
 u(a):=\sum_{n=1}^{\infty}\frac{m((P^n)',a)}{n^2},\quad
 \text{where}\quad m\bigl((P^n)',a\bigr):=\int_{\bP^1}\log[(P^n)'(z),a]_{\bP^1}\omega_{\FS}(z), \quad\text{on }\bP^1,
\end{gather*}
which yields 
that for every $a\in\bC\setminus E$ and every $C^2$-test function $\phi$ on $\bP^1$,
\begin{gather*}
 \biggl|\int_{\bP^1}\phi\rd\biggl(\frac{((P^n)')^*\delta_a-((P^n)')^*\omega_{\FS}}{d^n-1}\biggr)\biggr|
 \le\Bigl(\sup_{\bP^1}\Bigl|\frac{\rd\rd^c\phi}{\omega_{\FS}}\Bigr|\Bigr)
 \frac{-m((P^n)',a)}{d^n-1}
 =o(n^2d^{-n})\quad\text{as }n\to\infty
\end{gather*}
by integration by parts (using Russakovskii--Shiffman \cite[(26) and (29)]{RS97}).
On the other hand, the (qualitative) weak convergence of probability measures
\begin{gather}
 \lim_{n\to\infty}\frac{((P^n)')^*\omega_{\FS}}{d^n-1}=\mu_P\quad\text{on }\bP^1\label{eq:volderiv}
\end{gather}
itself is already highly non-trivial. 
Let us show the required quantitative version (by integration by parts)
\begin{gather}
\int_{\bP^1}\left|\frac{\log(1/[(P^n)',\infty]_{\bP^1})}{d^n-1}-g_P\right|\omega_{\FS}
=O\bigl(n^2d^{-n}\bigr)\quad\text{as }n\to\infty
\tag{\ref{eq:volderiv}'}\label{eq:volpotderiv}
\end{gather}
of \eqref{eq:volderiv}.
Here the escaping rate function
\begin{gather*}
 g_P:=\lim_{n\to\infty}\frac{\log(1/[P^n,\infty]_{\bP^1})}{d^n}\quad\text{on }\bP^1
\end{gather*}
of $P$ coincides with the Green function $g_{J(P),\infty}$ of 
$J(P)$, which is a regular compact subset in $\bC$, with pole $\infty$, and extends to $\bP^1$ by setting $\equiv 0$ on the filled-in Julia set $K(P)$ of $P$,
and then $\mu_P$ coincides with the harmonic measure for $K(P)$ with pole $\infty$. 
We denote by $\SAT(P)$ the set of all superattracting periodic points of $P$,
and recall that $\#\SAT(P)<+\infty$ and $\SAT(P)\cap\bC\subset\Int\{g_P=0\}$.

\begin{proof}[Proof of \eqref{eq:volpotderiv}]
In \cite[(3.2)]{OkuyamaDerivatives}, for every $n\in\bN$, we computed
\begin{multline}\label{eq:derivativeidentity}
\frac{\log|(P^n)'(z)|}{d^n-1}-g_P(z)
=\frac{1}{d^n-1}\int_{\bC}\Biggl(\sum_{j=0}^{n-1}\log[P^j(z),w]_{\bP^1}\Biggr)(\rd\rd^c\log|P'|)(w)\\
+\frac{d-1}{d^n-1}\sum_{j=0}^{n-1}\left(-\log[P^j(z),\infty]_{\bP^1}-g_P(P^j(z))\right)\\
+\left(-\int_{\bC}\log[w,\infty]_{\bP^1}(\rd\rd^c\log|P'|)(w)
+\log d+\log|a_d|\right)\frac{n}{d^n-1}\quad\text{on }\bP^1
\end{multline}
by the chain rule, where $a_d\in\bC\setminus\{0\}$ is the leading coefficient of $P$. 
Setting $C_P:=(d-1)\cdot\sup_{z\in\bP^1}\left|-\log[z,\infty]-g_P(z)\right|
 +(d-1)\cdot\sup_{w\in C(P)\cap\bC}\left|\log[w,\infty]\right|+\log d+|\log|a_d||\in\bR_{>0}$, by \eqref{eq:derivativeidentity} (and $g_P\ge 0$), we have
 \begin{align*}
 &\biggl|\frac{\log(1/[(P^n)',\infty]_{\bP^1})}{d^n-1}-g_P\biggr|
 \le\biggl|\frac{\log\max\{1,|(P^n)'|\}}{d^n-1}-g_P\biggr|+\frac{\log\sqrt{2}}{d^n-1}\\
 \le&\Biggl|\frac{\log\max\{1,|(P^n)'|\}}{d^n-1}
 -\max\biggl\{g_P-\int_{\bC}\biggl(\sum_{j=0}^{n-1}\frac{\log(1/[P^j,w]_{\bP^1})}{d^n-1}\biggr)
 (\rd\rd^c\log|P'|)(w),0\biggr\}\Biggr|\\
 &+\Biggl|\max\biggl\{g_P-\int_{\bC}\biggl(\sum_{j=0}^{n-1}\frac{\log(1/[P^j,w]_{\bP^1})}{d^n-1}\biggr)
 (\rd\rd^c\log|P'|)(w),0\biggr\}-g_P\Biggr|+\frac{\log\sqrt{2}}{d^n-1}\\
 \le& g_P-\max\biggl\{g_P-\int_{\bC}\biggl(\sum_{j=0}^{n-1}\frac{\log(1/[P^j,w]_{\bP^1})}{d^n-1}\biggr)
 (\rd\rd^c\log|P'|)(w),0\biggr\}+\frac{C_P\cdot n+\log\sqrt{2}}{d^n-1}\\
 =&\min\biggl\{g_P,\sum_{w\in C(P)\cap\bC}(\deg_w P-1)
 \sum_{j=0}^{n-1}\frac{\log(1/[P^j,w]_{\bP^1})}{d^n-1}\biggr\}
 +\frac{C_P\cdot n+\log\sqrt{2}}{d^n-1}\\
 \le&\sum_{w\in C(P)\cap\bC}\sum_{j=0}^{n-1}
 \min\biggl\{g_P,(\deg_w P-1)\frac{\log(1/[P^j,w]_{\bP^1})}{d^n-1}\biggr\}
 +\frac{C_P\cdot n+\log\sqrt{2}}{d^n-1}\\
 \le&\sum_{w\in C(P)\cap\bC}(\deg_w P-1)\sum_{j=0}^{n-1}
 \min\biggl\{g_P,\frac{\log(1/[P^j,w]_{\bP^1})}{d^n-1}\biggr\}
 +\frac{C_P\cdot n+\log\sqrt{2}}{d^n-1}\\
 =&\int_{\bC}\biggl(\sum_{j=0}^{n-1}\min\biggl\{g_P,\frac{\log(1/[P^j,w]_{\bP^1})}{d^n-1}\biggr\}\biggr)
 (\rd\rd^c\log|P'|)(w)+\frac{C_P\cdot n+\log\sqrt{2}}{d^n-1}\quad\text{on }\bP^1,
\end{align*}
so that by $C(P)\cap\SAT(P)\cap\bC\subset\Int\{g_P=0\}$ and
Theorem \ref{th:proximity} applied to each $w\in(C(P)\cap\bC)\setminus\SAT(P)$,
\begin{align*}
 &\int_{\bP^1}\biggl|\frac{\log(1/[(P^n)',\infty]_{\bP^1})}{d^n-1}-g_P\biggr|\omega_{\FS}\\
 \le&\int_{\bC}\Biggl(\sum_{j=0}^{n-1}\int_{\bP^1}
 \min\biggl\{g_P(z),\frac{\log(1/[P^j(z),w]_{\bP^1})}{d^n-1}\biggr\}
 \omega_{\FS}(z)\Biggr)(\rd\rd^c\log|P'|)(w)+\frac{O(n)}{d^n-1}\\
 \le&\int_{\bC\cap\SAT(P)}\Biggl(\sum_{j=0}^{n-1}\int_{\bP^1}
 \begin{cases}
 0 & \text{if }g_P(z)=0\\
 \displaystyle\frac{\log\bigl(1/\inf_{g_P(\zeta)>0}[\zeta,w]_{\bP^1}\bigr)}{d^n-1}
 & \text{if }g_P(z)>0
 \end{cases}
 \omega_{\FS}(z)\Biggr)(\rd\rd^c\log|P'|)(w)\\
 &-\frac{1}{d^n-1}\int_{\bC\setminus\SAT(P)}\Biggl(\sum_{j=0}^{n-1}m(P^j,w)\Biggr)
 (\rd\rd^c\log|P'|)(w)+\frac{O(n)}{d^n-1}\\
 \le&
 \#\bigl(C(P)\cap\SAT(P)\cap\bC\bigr)\cdot\frac{n\cdot O(1)}{d^n-1}
 +\#\bigl((C(P)\cap\bC)\setminus\SAT(P)\bigr)\cdot\frac{\sum_{j=0}^{n-1}O(j)}{d^n-1}\\
 &+\frac{O(n)}{d^n-1} =O(n^2d^{-n})\quad\text{as }n\to\infty,
\end{align*}
and we are done.
\end{proof}

In \cite{OkuyamaDerivatives}, we finally established a weaker 
$O(\eta^nd^{-n})$ version of the following parameter space analogue of Theorem \ref{th:derivative} for the monic centered degree $d>1$ unicritical polynomial family.

\begin{theorem}\label{th:unicrit}
 Let $P(\lambda,z)=P_\lambda(z)=z^d+\lambda\in\bC[z,\lambda]$
 be the monic centered unicritical polynomials family of degree $d>1$.
 Then there is a polar subset $E=E_P$ in $\bC\setminus\{0\}$ such that
 for every $a\in\bC\setminus E$ and every $C^2$-test function $\phi$ on $\bP^1$,
\begin{gather*}
 \int_{\bP^1}\phi(\lambda)\rd\left(\frac{((P_\lambda^n)'(\lambda))^*\delta_a}{d^n-1}-\mu_{C_d}\right)(\lambda)
=O(n^2d^{-n})\quad\text{as }n\to\infty.
\end{gather*}
Here $C_d:=\bigl\{\lambda\in\bC:C(P_\lambda)\cap\bC=\{0\}\subset K(P_\lambda)\bigr\}$ is the connectedness locus (or the degree $d$ version of the Mandelbrot set) of the family $P$ in the parameter space $\bC$ and $\mu_{C_d}$ is the harmonic measure for $C_d$ with pole $\infty$.
\end{theorem}
In this setting, taking into account Russakovskii--Shiffman \cite[(26) and (29)]{RS97},
the required estimate is
\begin{gather}
\int_{\bP^1}\biggl|\frac{\log(1/[(P_\lambda^n)'(\lambda),\infty]_{\bP^1})}{d^n-1}-g_{C_d,\infty}(\lambda)\biggr|\omega_{\FS}(\lambda)
=O(n^2d^{-n})\quad\text{as }n\to\infty.\label{eq:param}
\end{gather}
Here $g_{C_d,\infty}$ is the Green function of $C_d$,
which is also a regular compact subset in $\bC$, with pole $\infty$, 
and extends to $\bP^1$ by setting $\equiv 0$ on $C_d$.
We herewith include the proof of \eqref{eq:param} 
using \cite[(1.2)]{OkuyamaUnicritical2018} instead of Theorem \ref{th:proximity}.

\begin{proof}[Proof of \eqref{eq:param}]
Using the chain rule, we compute
\begin{gather*}
 \frac{\log|(P_\lambda^n)'(\lambda)|}{d^n-1}-g_{C_d,\infty}(\lambda)
 =\frac{d-1}{d^n-1}\sum_{j=1}^{n}
 \bigl(\log|P_\lambda^j(0)|-d^{j-1}g_{C_d,\infty}(\lambda)\bigr)
 +\frac{n\log d}{d^n-1}.
\end{gather*}
By \cite[(1.2)]{OkuyamaUnicritical2018},
\begin{gather*}
 \int_{\bP^1}\bigl|\log|P_\lambda^j(0)|-d^{j-1}g_{C_d,\infty}(\lambda)\bigr|
 \omega_{\FS}(\lambda)
 \le\frac{2\log d}{d-1}j+O(1)\quad\text{as }j\to\infty,
\end{gather*}
so that also using $g_{C_d,\infty}\ge0$, we have
\begin{align*}
 &\int_{\bP^1}\biggl|\frac{\log(1/[(P_\lambda^n)'(\lambda),\infty]_{\bP^1})}{d^n-1}-g_{C_d,\infty}(\lambda)\biggr|\omega_{\FS}(\lambda)\\
 \le&\int_{\bP^1}\biggl|\frac{\log\max\{1,|(P_\lambda^n)'(\lambda)|\}}{d^n-1}-g_{C_d,\infty}(\lambda)\biggr|\omega_{\FS}(\lambda)
 +\frac{\log\sqrt{2}}{d^n-1}\\
 \le&\int_{\bP^1}\biggl|\frac{\log|(P_\lambda^n)'(\lambda)|}{d^n-1}-g_{C_d,\infty}(\lambda)\biggr|\omega_{\FS}(\lambda)
 +\frac{\log\sqrt{2}}{d^n-1}\\
 \le&\frac{d-1}{d^n-1}\sum_{j=1}^{n}\int_{\bP^1}
 \bigl|\log|P_\lambda^j(0)|-d^{j-1}g_{C_d,\infty}(\lambda)\bigr|\omega_{\FS}(\lambda)
 +\frac{n\log d+\log\sqrt{2}}{d^n-1}\\
 \le&\frac{(\log d)n(n+1)+O(n)}{d^n-1} =O(n^2d^{-n})
\end{align*}
as $n\to\infty$.
\end{proof}

\begin{remark}
Indeed, the higher order derivative versions of Theorems
\ref{th:derivative} and \ref{th:unicrit} are true.
The proofs are based on other arguments instead of using Theorem \ref{th:proximity}
and \cite[(1.2)]{OkuyamaUnicritical2018}.
We postpone the details to our forthcoming paper.
\end{remark}

\def\cprime{$'$}

\end{document}